\documentclass[12pt]{amsart}

\usepackage{amssymb,amsmath,amsthm,enumerate,hyperref}

\newcommand{\excise}[1]{}

\newcommand{\KK}{\mathbf{K}}

\newcommand{\NN}{\mathbb{N}}
\newcommand{\OO}{\mathcal{O}}
\newcommand{\PP}{\mathcal{P}}

\newcommand{\RR}{\mathbb{R}}
\newcommand{\SR}{\mathrm{SR}}
\newcommand{\U}{\mathrm{U}_{\mathrm{SR}}}
\newcommand{\xx}{\mathbf{x}}
\newcommand{\ZZ}{\mathbb{Z}}

\DeclareMathOperator{\Cl}{Cl}

\DeclareMathOperator{\link}{link}
\DeclareMathOperator{\Neg}{neg}
\DeclareMathOperator{\Spec}{Spec}

\theoremstyle{plain}
\newtheorem{theorem}{Theorem}[section]
\newtheorem{lemma}[theorem]{Lemma}
\newtheorem{corollary}[theorem]{Corollary}
\newtheorem{proposition}[theorem]{Proposition}

\theoremstyle{definition}
\newtheorem{definition}[theorem]{Definition}

\begin{document}

\mbox{}
\vspace{-1.1ex}
\title{Cohomology of toric line bundles via simplicial Alexander duality}
\author{Shin-Yao Jow}
\address{Department of Mathematics,
University of Pennsylvania,
Philadelphia, PA 19104}
\email{\texttt{jows@math.upenn.edu}}
\date{23 December 2010}

\begin{abstract}
We give a rigorous mathematical proof for the validity of the toric sheaf cohomology algorithm conjectured in the recent paper by R. Blumenhagen, B. Jurke, T. Rahn, and H. Roschy (\texttt{arXiv:1003.5217}). We actually prove not only the original algorithm but also a speed-up version of it. Our proof is independent from (in fact appeared earlier on the arXiv than) the proof by H. Roschy and T. Rahn (\texttt{arXiv:1006.2392}), and has several advantages such as being shorter and cleaner and can also settle the additional conjecture on ``Serre duality for Betti numbers'' which was raised but unresolved in \texttt{arXiv:1006.2392}.
\end{abstract}


\maketitle

\section{Introduction} 

Recently in \cite{BJRR} a new algorithm for computing the cohomology groups of line bundles on a toric variety was conjectured, which the authors observed to be more efficient than previously known algorithms such as those described in \cite[\S 3.5]{Ful}, \cite{EMS}, or \cite[\S 9.1]{CLS}. However apart from having no proof, \cite{BJRR} also gave a somewhat vague treatment of a crucial step in the algorithm involving certain ``remnant cohomology''. The original goal of this article was to provide the first rigorous mathematical proof for the validity of the algorithm, as well as a precise procedure to compute the ``remnant cohomology''. A week after the first version of this article appeared on the arXiv, an independent and alternative proof was given in \cite{RR}, in the last section of which the authors further conjectured a ``Serre duality for Betti numbers'' that will speed up the algorithm even more. It turns out that this additional conjecture also follows quite easily from our original method of proof, possibly due to the fact that our language is more topological, which makes our proof cleaner and shorter than the more algebraic proof in \cite{RR}, so it is easier for us to have a clearer view. Thus in this updated version, the main theorem (Theorem~\ref{t:main}) is strengthened to incorporate this additional conjecture, and its immediate corollary (Corollary~\ref{c:algorithm}) is likewise improved to give the speed-up algorithm.    

The basic framework we will use to explain the algorithm in \cite{BJRR} is the results in \cite[Section~2]{EMS} which connect toric sheaf cohomology with local cohomology over a polynomial ring (the Cox ring). Let $X$ be a $d$-dimensional toric variety over a field $\KK$ associated to a fan $\Delta$ in $N\cong\ZZ^d$. We denote by $\Delta(1)$ the set of 
one-dimensional cones (i.e. rays) in $\Delta$, and assume that $\Delta(1)$ spans $N_{\RR}=N\otimes_{\ZZ}\RR$. The Cox ring of $X$ (\cite{Aud}, \cite{Mus}, \cite{Cox})  is the polynomial ring $S$ in the set of variables indexed by the rays in $\Delta$: \[
          S=\KK[x_{\rho}\mid \rho\in \Delta(1)]. \]
One can consider on $S$ the usual multigrading on a polynomial ring, which in this case is  the grading by $\ZZ^{\Delta(1)}$, the free abelian group with a basis indexed by the rays (of course in $S$ only the $\NN^{\Delta(1)}$-graded pieces are nonzero, but later we will also need to consider localizations of $S$ which invert some of the variables). There is also an important squarefree monomial ideal $B$ of $S$, called the irrelevant ideal, defined as follows. For each cone $\sigma$ in the fan $\Delta$, let $\sigma(1)=\{\rho\in\Delta(1) \mid \rho\subset\sigma\}$. Then \[
 B=\langle\prod_{\rho\notin\sigma(1)}x_\rho \mid \sigma\in\Delta \rangle. \]
The relations between $X$ and $(S,B)$ are very similar to those between a projective space and its homogeneous coordinate ring with the usual irrelevant ideal. For example it was shown in \cite{Cox} that $X$ is a quotient of the complement of the algebraic subset defined by $B$ in the affine space $\Spec S$. Moreover, under this quotient map, the image of the coordinate hyperplane $(x_\rho=0) \subset \Spec S$ is the torus-invariant prime divisor $D_\rho\subset X$ corresponding to the ray $\rho$. In view of this, it is natural to endow $S$ with another coarser grading by the class group of $X$: one simply sets the degree of a monomial $\prod_{\rho\in\Delta(1)}x_\rho^{a_\rho}$ to be the divisor class $[\sum_{\rho\in\Delta(1)} a_\rho D_\rho]$ in $\Cl(X)$. In the special case when $X$ is a projective space, this is simply the grading by the total degree. If $F$ is a 
$\Cl(X)$-graded module over $S$, then an associated quasi-coherent sheaf $\widetilde{F}$ on $X$ can be constructed in a way similar to the projective space case, and it is also true that every coherent sheaf on $X$ is of the form $\widetilde{F}$ for some finitely generated $F$ (\cite[Proposition~3.3]{Cox} for simplicial $X$; \cite[Theorem~2.1]{EMS}). Using \v{C}ech cohomology, it was shown \cite[Proposition~2.3]{EMS} that for $i\ge 1$, there is an isomorphism of $\Cl(X)$-graded $S$-modules \[
 \bigoplus_{\alpha\in\Cl(X)} H^i\left(\widetilde{F(\alpha)}\right) \cong  H^{i+1}_B(F),\]
where $F(\alpha)$ is the shifted module $F(\alpha)_\beta=F_{\alpha+\beta}$, and $H^{i+1}_B(F)$ is the $(i+1)^{\text{th}}$ local cohomology of $F$ with support on $B$ 
\cite[Definition~13.1]{MS}. This isomorphism is the starting point for studying sheaf cohomology on a toric variety both in \cite{EMS} and in this paper. We will focus on the case $F=S$, since $\widetilde{S(\alpha)}\cong \OO_X(D)$ if $D$ is a divisor with divisor class $\alpha$, so $H^{i+1}_B(S)$ already encodes the $i^{\text{th}}$ sheaf cohomology of all line bundles on $X$. Following \cite{EMS}, we use the notation $H^i_\ast(\OO_X)$ to denote the left-hand side of the above isomorphism when $F=S$. So we have 
\begin{equation} \label{e:1}
 H^i_\ast(\OO_X)=\bigoplus_{\alpha\in\Cl(X)} H^i\left(\widetilde{S(\alpha)}\right) \cong  H^{i+1}_B(S).
\end{equation}

The local cohomology $H^{i+1}_B(S)$ can be computed by a \v{C}ech complex whose terms involve only localizations of $S$ inverting some of the variables (\cite[Theorem~13.7]{MS}; see also Section~\ref{s:localcoh} of this paper). Consequently $H^{i+1}_B(S)$, and hence $H^i_\ast(\OO_X)$, enjoy the finer $\ZZ^{\Delta(1)}$-grading. It is known that each of the graded pieces $H^i_\ast(\OO_X)_p$, $p\in \ZZ^{\Delta(1)}$, depends only on which coordinates of $p$ are negative (\cite[Theorem~2.4]{EMS}; see also Proposition~\ref{p:3.1}). More precisely, if one defines \[
   \Neg(p)=\{ \rho\in\Delta(1) \mid \text{the $\rho$ coordinate of $p$ is negative}\}, \]
then $H^i_\ast(\OO_X)_p$ is canonically isomorphic to $H^i_\ast(\OO_X)_q$ for any $p,q\in \ZZ^{\Delta(1)}$ such that $\Neg(p)=\Neg(q)$. Hence it makes sense to define $H^i_\ast(\OO_X)_I$ for a subset $I\subset \Delta(1)$, by setting it to be $H^i_\ast(\OO_X)_p$ for any $p\in \ZZ^{\Delta(1)}$ such that $\Neg(p)=I$. One then wants to know which subsets $I$ give rise to nonzero graded pieces $H^i_\ast(\OO_X)_I$, and how these graded pieces can be computed. Our main theorem provides a simple combinatorial answer to this question when $X$ is a simplicial projective toric variety. To state it, it will be convenient to define two collections of subsets of $\Delta(1)$: first following the notation in \cite{BJRR} we define \[
 \SR=\Big\{\hat{J}\subset\Delta(1)\Bigm| \text{\parbox{9cm}{$\hat{J}$ does not span a cone in $\Delta$, but every proper subset of $\hat{J}$ spans a cone in $\Delta$.}}\Big\}. \]
The squarefree monomials of the form $\prod_{\rho\in\hat{J}} x_\rho$, $\hat{J}\in\SR$ are precisely the minimal generators of the so-called Stanley-Reisner ideal of $\Delta$. Then we define $\U$ to be the collection of all subsets of $\Delta(1)$ which can be expressed as a union $\hat{J_1} \cup \cdots \cup \hat{J_m}$ for some $\hat{J_1},\ldots,\hat{J_m}$ in $\SR$.

\begin{theorem} \label{t:main}
 Let $X$ be a simplicial projective toric variety of dimension $d$ associated to a fan $\Delta$. Let $I$ be a subset of $\Delta(1)$ and let $i\ge 1$ be a positive integer. Then
 \begin{enumerate}[\upshape (a)]
  \item
   $H^i_\ast(\OO_X)_I = 0$ unless $I\in \U$.
  \item
   Let $\hat{I}$ denote the complement of\/ $I$ in $\Delta(1)$. If\/ $i\ne d$, then $H^i_\ast(\OO_X)_I$ is naturally dual to $H^{d-i}_\ast(\OO_X)_{\hat{I}}$ (as $\KK$-vector spaces). Combined with part~(a) this implies in particular that $H^i_\ast(\OO_X)_I = 0$ unless both $I$ and $\hat{I}$ are in $\U$.
  \item
    If $I\in \U$, let $\hat{J_1},\ldots,\hat{J_m}$ be all of the elements in $\SR$ that are contained in $I$. Define $\Lambda_I$ to be the following abstract simplicial complex on the vertex set $\{1,\ldots,m\}$:  \[
    \Lambda_I=\bigg\{K\subset\{1,\ldots,m\} \biggm| \bigcup_{k\in K} \hat{J_k} \ne I \bigg\}. \]
   Then there is a natural isomorphism   \[
     H^i_\ast(\OO_X)_I \cong \widetilde{H}_{|I|-i-2}(\Lambda_I),  \]
   where the right-hand side is the reduced homology of $\Lambda_I$ (with coefficients in $\KK$).
 \end{enumerate}
\end{theorem}

We remark that part~(b) is the ``Serre duality for Betti numbers'' conjectured in the last section of \cite{RR}, while part~(c) is essentially the ``remnant'' cohomology $\mathcal{H}^i(\mathcal{Q})$ in \cite{BJRR}.\footnote{In our notations the ``remnant'' cohomology $\mathcal{H}^i(\mathcal{Q})$ is really the (reduced) relative homology of the full simplex on $\{1,\ldots,m\}$ modulo $\Lambda_I$, hence is essentially the same as the reduced homology of $\Lambda_I$ since the full simplex is contractible.}

 As an immediate corollary we obtain the following speed-up version of the algorithm in \cite{BJRR}:

\begin{corollary} \label{c:algorithm}
 Given a line bundle $L$ on $X$ and an integer $0<i<d$,\footnote{The cases $i=0$ or $d$ are known to be easy. When $i=0$ the cohomology can be computed by counting lattice points in a certain polytope: see for example \cite[p.66]{Ful}. The case $i=d$ can be reduced to the case $i=0$ by the usual Serre duality.} we have \[
h^i(X,L)=\sum_{I} \#\bigg\{p\in\ZZ^{\Delta(1)} \biggm| \Neg(p)=I \text{ and } \Big[\sum_{\rho\in\Delta(1)}p_\rho D_\rho\Big]=L\bigg\}\cdot \dim \widetilde{H}_{|I|-i-2}(\Lambda_I), \]
where the sum is over all\/ $I\subset\Delta(1)$ such that both $I$ and $\hat{I}$ are in $\U$. (In the original algorithm in \cite{BJRR} the sum is over all\/ $I$ in $\U$.)
\end{corollary}

We remark that from a theoretic point of view, the efficiency of this algorithm seems to come primarily from the vanishing in Theorem~\ref{t:main}~(a)(b), which greatly reduces the amount of homological computation one needs to perform. We refer the readers to the Introduction in \cite{RR} for more details on implementation of the algorithm and how it compares with other previously known algorithms.

The organization of the rest of this article is as follows: Section~\ref{s:topological} contains some topological preliminaries we will need in our proof, including our primary tool \emph{the simplicial Alexander duality} (Corollary~\ref{c:SAD}) and the topological source of the vanishing in Theorem~\ref{t:main}~(a) (Lemma~\ref{l:2.4}). In Section~\ref{s:localcoh} we review the method for computing the local cohomology $H^i_B(S)$ using the natural cellular resolution of $S/B$ supported on the moment polytope of $X$, as done in \cite[\S 3]{MS04}. After these preparations, we present our rather short proof for Theorem~\ref{t:main} in Section~\ref{s:proof}.

\section*{Acknowledgements}

The author would like to thank Ezra Miller for several useful comments on the first draft of this article.

\section{Combinatorial topological preliminaries} \label{s:topological}

In this section we collect some results from simplicial topology which will come into the proof of the main theorem, most notably the simplicial Alexander duality. Throughout this section $\Gamma$ will denote an (abstract) simplicial complex on a finite vertex set $V$, i.e. $\Gamma$ is a collection of subsets of $V$, such that if $\sigma\in\Gamma$ and $\tau\subset\sigma$ then $\tau\in\Gamma$. Each $\sigma\in\Gamma$ is called a simplex or a face of $\Gamma$. The underlying topological space of $\Gamma$ is denoted by $\|\Gamma\|$. Also recall that the link of a face $\sigma$ in $\Gamma$ is the subcomplex \[
  \link_\Gamma \sigma=\{\tau\in \Gamma \mid \tau\cup\sigma \in
  \Gamma\text{ and } \tau\cap\sigma=\varnothing\}. \]
Given any subset $\sigma\subset V$ (not necessarily a face of $\Gamma$), the notation $\hat{\sigma}$ will denote its complement $\hat{\sigma}=V\setminus\sigma$, and the notation $\Gamma_{\le\sigma}$ will denote the simplicial complex consisting of every face of $\Gamma$ that is contained in $\sigma$: \[
 \Gamma_{\le\sigma}=\{\tau\in\Gamma \mid \tau\subset\sigma \}. \] 

\begin{definition}
 The \emph{Alexander dual} of $\Gamma$, denoted by $\Gamma^*$, is the following simplicial complex on the vertex set $V$: \[
 \Gamma^*=\{ \sigma\subset V \mid \hat{\sigma}\notin \Gamma \}. \]
\end{definition}  

\begin{theorem}[Simplicial Alexander duality] \label{t:SAD}
For every integer $j$ there is an isomorphism \[ 
\widetilde{H}_j(\Gamma^*)\cong \widetilde{H}^{|V|-3-j}(\Gamma). \]
\end{theorem}

 A self-contained proof from first principles was given in \cite{BT}, and a proof using homological algebra can be found in \cite[Section~5.1]{MS}. The connection to the topological Alexander duality
$\widetilde{H}_j(S^{n-2}\setminus A)\cong \widetilde{H}^{n-3-j}(A)$ is that if $|V|=n$, then the simplicial complex consisting of all proper subsets of $V$ has an underlying space homeomorphic to $S^{n-2}$, and it contains the closed subcomplex $\Gamma$ whose underlying space plays the role of $A$. Furthermore it can be shown that $S^{n-2}\setminus \|\Gamma\|$ is homotopy equivalent to $\|\Gamma^*\|$: see Proposition~2.27 of the book \cite{BP}. What we will use in fact is the following equivalent version given in Proposition~2.29 of that same book:  
 
\begin{corollary}[Simplicial Alexander duality---alternative version] \label{c:SAD}
 If $\sigma\in\Gamma^*$, then \[
  \widetilde{H}_j(\link_{\Gamma^*}\sigma)\cong \widetilde{H}^{|V|-3-j-|\sigma|}(\Gamma_{\le\hat{\sigma}}). \] 
\end{corollary}

\begin{proof}
This follows from Theorem~\ref{t:SAD} because $\link_{\Gamma^*}\sigma$ is the Alexander dual of $\Gamma_{\le\hat{\sigma}}$ when both are viewed as simplicial complexes on the vertex set $\hat\sigma$. (On the other hand setting $\sigma=\varnothing$ recovers Theorem~\ref{t:SAD}, so the two statements are actually equivalent.)
\end{proof}

We will also need the following two simple lemmas:

\begin{lemma} \label{l:2.3}
 For any subset $\sigma\subset V$, it holds that $\|\Gamma_{\le\hat{\sigma}}\|$ is a deformation retract of $\|\Gamma\| \setminus \|\Gamma_{\le\sigma}\|$.
\end{lemma}

\begin{proof}
Let $\Gamma_{\sigma,\hat{\sigma}}$ be the set consisting of every face of $\Gamma$ which is neither contained in $\sigma$ nor in $\hat{\sigma}$. Each $\tau\in\Gamma_{\sigma,\hat{\sigma}}$ can be written as a proper disjoint union $\tau=\tau_{\sigma} \sqcup \tau_{\hat{\sigma}}$, where $\tau_{\sigma}=\tau\cap\sigma$ and $\tau_{\hat{\sigma}}=\tau\cap\hat{\sigma}$. Observing that \[
 \|\Gamma\| \setminus \|\Gamma_{\le\sigma}\|=\|\Gamma_{\le\hat{\sigma}}\| \cup
   \bigg( \bigcup_{\tau\in\Gamma_{\sigma,\hat{\sigma}}} \|\tau\|\setminus\|\tau_{\sigma}\|\bigg), \]
the lemma thus follows since each $\|\tau\|\setminus\|\tau_{\sigma}\|$ can be deformation retracted to $\|\tau_{\hat{\sigma}}\|$.
\end{proof}

\begin{lemma} \label{l:2.4}
Given $\sigma\in\Gamma$, let $\tau_1,\ldots,\tau_m\in\Gamma$ be the maximal faces of $\Gamma$ containing $\sigma$. If $\tau_1\cap\cdots\cap\tau_m\supsetneqq\sigma$, then 
$\|\link_{\Gamma}\sigma\|$ is contractible.
\end{lemma}

\begin{proof}
The maximal faces of $\link_{\Gamma}\sigma$ are precisely $\tau_1\setminus\sigma,\ldots,\tau_m\setminus\sigma$, and by assumption they have a nonempty intersection. Hence the lemma follows.
\end{proof}

\section{Toric local cohomology} \label{s:localcoh}

Let $X$ be a $d$-dimensional simplicial projective toric variety associated to a fan $\Delta$ in $N\cong\ZZ^d$. Let $\Delta(i)$ denote the set of $i$-dimensional cones in $\Delta$. Recall from the Introduction that the Cox ring of $X$ is the polynomial ring $S=\KK[x_{\rho}\mid \rho\in \Delta(1)]$, and the irrelevant ideal is 
$B=\langle\prod_{\rho\notin\sigma(1)}x_\rho \mid \sigma\in\Delta \rangle$. In this section we review the method for computing the local cohomology $H^i_B(S)$ using the natural cellular resolution of $S/B$ supported on the moment polytope of $X$. Our main reference is \cite{MS} (see also \cite[\S 3]{MS04} or \cite[Example~6.6]{Mil}).

A common way to compute local cohomology is to use the usual \v{C}ech complex \cite[Definition~13.5]{MS}. This is often not the most efficient complex one can use. In fact any free resolution $\mathcal{F}_\bullet$ of $S/B$ gives rise to a generalized \v{C}ech complex $\check{\mathcal{C}}_{\mathcal{F}}^\bullet$ \cite[Definition~13.28]{MS}, which can be used to compute $H^i_B(S)$ \cite[Theorem~13.31]{MS}. The usual \v{C}ech complex came from the Taylor resolution \cite[\S 4.3.2]{MS}, which tends not to be minimal. In our present toric situation, the best choice for $\mathcal{F}_\bullet$ is the natural cellular resolution supported on the moment polytope of $X$, which turns out to be minimal \cite[\S 4.3.6]{MS}. We will now briefly describe this resolution. The readers who are unfamiliar with the basic theory of cellular resolutions could consult \cite[\S 4.1]{MS} before going on.

First we need to describe the moment polytope of $X$. In the theory of projective toric varieties this is a well-known polytope in the dual space of $N$, but since we are only concerned about its combinatorial structure, it suffices to know that the set of all of its faces is in an inclusion-reversing bijection with $\Delta$. Hence we can identify the set of $i$-dimensional faces of the moment polytope with the set $\Delta(d-i)$, and label each face $\sigma$ with the monomial $\xx^{\hat{\sigma}}\overset{\text{def}}{=}\prod_{\rho\notin\sigma(1)}x_\rho$. The moment polytope together with this labeling fits the definition of a labeled cell complex \cite[Definition~4.2]{MS}, and   gives rise to the following cellular free resolution of $S/B$: \[
\mathcal{F}_\bullet \colon \ 0\longleftarrow S \longleftarrow \bigoplus_{\sigma\in\Delta(d)} S\,\xx^{\hat{\sigma}} \longleftarrow \bigoplus_{\sigma\in\Delta(d-1)} S\,\xx^{\hat{\sigma}} \longleftarrow \cdots \longleftarrow \bigoplus_{\sigma\in\Delta(0)} S\,\xx^{\hat{\sigma}} \longleftarrow 0, \]
where the homomorphism from a direct summand $S\,\xx^{\hat{\sigma}}$ of $\mathcal{F}_i$ to a direct summand $S\,\xx^{\hat{\tau}}$ of $\mathcal{F}_{i-1}$ is $0$ if $\tau\not\supset\sigma$, and is the inclusion map times $\pm 1$ if $\tau\supset\sigma$. The sign is determined by arbitrary but fixed orientations for the cones in $\Delta$. 
The corresponding generalized \v{C}ech complex $\check{\mathcal{C}}_{\mathcal{F}}^\bullet$ is by definition \[
\check{\mathcal{C}}_{\mathcal{F}}^\bullet \colon \ 0\longrightarrow S \longrightarrow \bigoplus_{\sigma\in\Delta(d)} S\Big[\frac{1}{\xx^{\hat{\sigma}}}\Big] \longrightarrow \bigoplus_{\sigma\in\Delta(d-1)} S\Big[\frac{1}{\xx^{\hat{\sigma}}}\Big] \longrightarrow \cdots \longrightarrow   \bigoplus_{\sigma\in\Delta(0)} S\Big[\frac{1}{\xx^{\hat{\sigma}}}\Big] \longrightarrow 0, \]
where the homomorphism between two direct summands $S[\frac{1}{\xx^{\hat{\tau}}}]$ and $S[\frac{1}{\xx^{\hat{\sigma}}}]$ for each codimension one $\sigma\subset\tau$ is again the inclusion map times the sign determined by orientations.

By \cite[Theorem~13.31]{MS} we have $H^i_B(S)=H^i(\check{\mathcal{C}}_{\mathcal{F}}^\bullet)$. Since $\check{\mathcal{C}}_{\mathcal{F}}^\bullet$ is $\ZZ^{\Delta(1)}$-graded, so is $H^i_B(S)$. In fact given any $p\in\ZZ^{\Delta(1)}$ we can describe the graded piece $H^i_B(S)_p$ explicitly as follows. Let $I\subset\Delta(1)$ be the set of places where $p$ has negative coordinates, i.e. \[
   I=\Neg(p)\overset{\text{def}}{=}\{ \rho\in\Delta(1) \mid \text{the $\rho$ coordinate of $p$ is negative}\}. \]
Let $\hat{I}=\Delta(1)\setminus I$. Then the graded piece $S[\frac{1}{\xx^{\hat{\sigma}}}]_p$ is $0$ if $\sigma(1)\not\subset\hat{I}$, and is $\KK\xx^p$ if $\sigma(1)\subset\hat{I}$, where $\xx^p$ denotes the monomial whose exponent vector is $p$. Taking the graded piece of degree $p$ of every term in $\check{\mathcal{C}}_{\mathcal{F}}^\bullet$ and then computing the cohomology, one obtains the following description of $H^i_B(S)_p$, which is essentially the same as \cite[Theorem~2.7]{EMS} or \cite[Proposition~3.2]{MS04}:

\begin{proposition} \label{p:3.1}
Let $\PP$ be the abstract simplicial complex on the vertex set $\Delta(1)$ given by  $\PP=\{\sigma(1) \mid \sigma\in\Delta\}$. Then for every $p\in\ZZ^{\Delta(1)}$ and $i\ge 2$ there is an isomorphism \[
 H^i_B(S)_p\cong \widetilde{H}_{d-i}(\PP_{\le\hat{I}}),  \]
where $I=\Neg(p)$. In particular $H^i_B(S)_p$ only depends on $\Neg(p)$.
\end{proposition}

\section{Proof of Theorem} \label{s:proof}

We now present the proof of our main theorem.

\begin{proof}[Proof of Theorem~\ref{t:main}]
By \eqref{e:1} we have $H^i_\ast(\OO_X)_I \cong  H^{i+1}_B(S)_I$, and by Proposition~\ref{p:3.1} we have $H^{i+1}_B(S)_I\cong \widetilde{H}_{d-1-i}(\PP_{\le\hat{I}})$. Lemma~\ref{l:2.3} then implies that \[ 
\widetilde{H}_{d-1-i}(\PP_{\le\hat{I}})\cong \widetilde{H}_{d-1-i}(\|\PP\|\setminus\|\PP_{\le I}\|). \]
Since $\|\PP\|$ is homeomorphic to $S^{d-1}$, the topological Alexander duality gives \[ \widetilde{H}_{d-1-i}(\|\PP\|\setminus\|\PP_{\le I}\|)\cong \widetilde{H}^{i-1}(\PP_{\le I}). \]
Now there are two possible cases for $I$:
\begin{description}
\item[Case 1] $\hat{I}\notin\PP^*$. This is equivalent to $I\in\PP$, which implies that $\PP_{\le I}$ is the full simplex on $I$, so $\widetilde{H}^{i-1}(\PP_{\le I})=0$.
\item[Case 2] $\hat{I}\in\PP^*$. In this case we can apply Corollary~\ref{c:SAD} to get \[
  \widetilde{H}^{i-1}(\PP_{\le I})\cong \widetilde{H}_{|I|-2-i}(\link_{\PP^*}\hat{I}). \]
 Let $J_1,\ldots,J_m$ be the maximal faces of $\PP^*$ containing $\hat{I}$. There are two further subcases: either $J_1\cap\cdots\cap J_m\supsetneqq\hat{I}$ or $J_1\cap\cdots\cap J_m=\hat{I}$. If $J_1\cap\cdots\cap J_m\supsetneqq\hat{I}$ then $\widetilde{H}_{|I|-2-i}(\link_{\PP^*}\hat{I})=0$ by Lemma~\ref{l:2.4}. Hence $H^i_\ast(\OO_X)_I=0$ unless $J_1\cap\cdots\cap J_m=\hat{I}$, or equivalently $I=\hat{J_1}\cup\cdots\cup \hat{J_m}$. This proves part~(a) since \[
 \SR=\{\hat{J}\mid \text{$J$ is a maximal face of $\PP^*$}\}. \]  
Now if $J_1\cap\cdots\cap J_m=\hat{I}$, then $J_1\setminus \hat{I}, \ldots,J_m\setminus \hat{I}$ are precisely the maximal faces of $\link_{\PP^*}\hat{I}$, and the nerve they form (in the sense of \cite{Gru}) is precisely $\Lambda_I$. Hence \[
 \widetilde{H}_{|I|-2-i}(\link_{\PP^*}\hat{I})\cong \widetilde{H}_{|I|-2-i}(\Lambda_I) \]
 by \cite[Theorem~10]{Gru}, which proves part~(c).
\end{description}
Finally to see part~(b), note that the series of isomorphisms in the beginning part of the proof showed that \[
 H^i_\ast(\OO_X)_I \cong \widetilde{H}_{d-1-i}(\PP_{\le\hat{I}}) \cong \widetilde{H}^{i-1}(\PP_{\le I}). \]
So $H^i_\ast(\OO_X)_I\cong \widetilde{H}^{i-1}(\PP_{\le I})$, and 
$H^{d-i}_\ast(\OO_X)_{\hat{I}}\cong \widetilde{H}_{i-1}(\PP_{\le \hat{\hat{I}}})=\widetilde{H}_{i-1}(\PP_{\le I})$ if $i\ne d$, hence they are dual to each other.
\end{proof}


\end{document}